\documentclass[conference]{IEEEtran}
\IEEEoverridecommandlockouts
\usepackage{cite}
\usepackage{amsmath,amssymb,amsfonts}
\usepackage{algorithmic}
\usepackage{graphicx}
\usepackage{textcomp}
\usepackage{xcolor}

\usepackage{amsthm}
\usepackage{booktabs}
\usepackage{bbm}
\usepackage{caption}
\usepackage{bbold}
\usepackage{mathtools}

\def\BibTeX{{\rm B\kern-.05em{\sc i\kern-.025em b}\kern-.08em
    T\kern-.1667em\lower.7ex\hbox{E}\kern-.125emX}}
\begin{document}

	\title{Data-driven Distributionally Robust MPC: An indirect feedback approach}
\author{Christoph Mark and Steven Liu
	\thanks{Institute  of  Control  Systems,  Department  of  Electrical  and  Computer
		Engineering, University of Kaiserslautern, Erwin-Schrödinger-Str. 12, 67663
		Kaiserslautern, Germany, {\tt\small mark|sliu@eit.uni-kl.de}}%
}

\maketitle

\newtheorem{theorem}{Theorem}
\newtheorem{corollary}{Corollary}
\newtheorem{lemma}{Lemma}
\newtheorem{definition}{Definition}
\newtheorem{remark}{Remark}
\newtheorem{assum}{Assumption}

\maketitle
\begin{abstract}
	This paper presents a distributionally robust stochastic model predictive control (SMPC) approach for linear discrete-time systems subject to unbounded and correlated additive disturbances. We consider hard input constraints and state chance constraints, which are approximated as distributionally robust (DR) Conditional Value-at-Risk (CVaR) constraints over a Wasserstein ambiguity set. The computational complexity is reduced by resorting to a tube-based MPC scheme with indirect feedback, such that the error scenarios can be sampled offline. Recursive feasibility is guaranteed by softening the CVaR constraint. The approach is demonstrated on a four-room temperature control example.
\end{abstract}


\section{Introduction}
Stochastic MPC is an advanced predictive control strategy that can be roughly classified into analytical approximation methods and randomized methods \cite{farina2016stochastic}. 
On the one hand, analytical approximation methods typically rely on distributional information such as the mean and variance of the disturbance and/or its probability distribution in order to cast the stochastic control problem into a deterministic one \cite{farina2016stochastic}. 
On the other hand, randomized approaches utilize samples/scenarios of the disturbance to approximate the stochastic control problem via a sample-average approximation (SAA). Probabilistic closed-loop guarantees can be established using tools from scenario optimization \cite{hewing2019scenario}. The main benefit of randomized methods lies in the ability to cope with generic disturbances, whereas analytical methods require assumptions on the disturbance.

This paper presents an indirect-feedback distributionally robust SMPC (DR-SMPC) scheme for chance constrained linear systems. By resorting to a scenario-based tube formulation we achieve a reduced online complexity of the MPC optimization problem compared to a classical scenario-based SMPC. Furthermore, the optimization problem is robustified against distributional uncertainty by introducing techniques from Wasserstein Distributionally Robust Optimization (DRO) \cite{kuhn2019wasserstein}. Since chance constraints are generally non-convex, we approximate them as Conditional Value-at-Risk (CVaR) constraints \cite{nemirovski2007convex} and evaluate them under the worst-case distribution of the Wasserstein ambiguity set. In fact, for many control tasks the use of CVaR constraints appears more natural than chance constraints, i.e. chance constraints only penalize the frequency of constraint violations, while the CVaR additionally penalizes the severity of the violation. In addition, the scenario-based tube computation allows for the use of nonlinear tube controllers, e.g. saturated controllers that enable the treatment of hard input constraints.
Lastly, we ensure recursive feasibility by softening the distributionally robust CVaR constraints with slack variables, which was similarly done in \cite{lu2020soft} for the case of affine disturbance feedback policies and moment-based ambiguity sets.

\subsubsection*{Related work}
The authors of \cite{hewing2019scenario} proposed an indirect feedback SMPC approach with scenario-based Probabilistic Reachable Sets (PRS). By decoupling the nominal from the error system, recursive feasibility in case of correlated and unbounded disturbances was achieved. In \cite{MARK20207136}, we recently proposed a similar MPC scheme with distributionally robust PRS. 
Coulson et al. \cite{coulson2020distributionally} presented a data-driven DR-MPC framework for CVaR constrained systems described by data matrix time series. Similar to our approach a Wasserstein ambiguity set was used to robustify the MPC optimization problem against sampling errors. Coppens and Patrinos \cite{Coppens} recently proposed a distributionally robust MPC framework for chance constrained stochastic systems under conic representable ambiguity sets with independent and identically distributed process noise. A similar problem setting was considered by \cite{zhong2021data} for Wasserstein ambiguity sets.
\subsubsection*{Outline}This paper is organized as follows: In Section \ref{sec:preliminaries} we introduce the problem description, the Wasserstein ambiguity set and the distributionally robust cost and chance constraints. Section \ref{sec:DR-Smpc} is dedicated to the indirect feedback SMPC scheme and the tractable approximation of the cost and CVaR constraints. The section ends with our main results on recursive feasibility, closed-loop input and predictive state constraint satisfaction. The paper closes with an example of a four room temperature control task in Section \ref{sec:example}.
\section{Preliminaries}
\subsubsection*{Notation}
 The Pontryagin difference between two polytopic sets $\mathbb{A}$ and $\mathbb{B}$ is given by $\mathbb{A} \ominus \mathbb{B} = \{ a \in \mathbb{A} : a+b \in \mathbb{A}, \forall b \in \mathbb{B} \}$.
Positive definite and semidefinite matrices are indicated as $A>0$ and $A\geq0$, respectively. We denote the set of nonnegative real numbers as $\mathbb{R}_{\geq 0}$. For an event $E$ we define the probability of occurrence as $\mathbb{P}(E)$. A random variable $w$ following a distribution $\mathbb{Q}$ is denoted as $w \sim \mathbb{Q}$, where the expected value w.r.t. $\mathbb{Q}$ is given by $\mathbb{E}_{\mathbb{Q}}(w)$. 
The positive part of a real-valued function $f$ is given by $(f(x))_+ = \max(0, f(x))$.
For two vectors $x,y \in \mathbb{R}^n$ we define the dual norm of a norm $\Vert \cdot \Vert$ as $\Vert x \Vert_* \coloneqq \sup_{\Vert y \Vert \leq 1} x^\top y$. The convex conjugate of a function $f: \mathbb{X} \to \mathbb{R}$ is denoted as $f^*(\theta) \coloneqq \sup_{x \in \mathbb{X}} \theta^\top x - f(x)$. The effective domain of $f$ is defined to be the set $\text{dom}(f) = \{ x \in \mathbb{X} | f(x) < \infty \}$.
The superscript $ \hat{\cdot}$ denotes a data dependent quantity.

\label{sec:preliminaries}
\subsection{Problem description}
We consider discrete linear time-invariant systems of the form
\begin{align}
	x(k+1) = A x(k) + B u(k) + \bar{w}(k) + w(k), \label{eq:dynamic}
\end{align}
where $x \in \mathbb{R}^n$ and $u \in \mathbb{R}^m$ denote the state and input vectors and $A \in \mathbb{R}^{n \times n}$ and $B \in \mathbb{R}^{n \times m}$ are matrices of conformal dimension. The additive disturbance consists of a known part $\bar{w}$ that is bounded in a compact set $\bar{w} \in \bar{\mathcal{W}}$, e.g. the mean, and a stochastic part $w(k)$ that follows an unknown distribution. For the sake of simplicity we assume that the pair $(A,B)$ is controllable and perfect state measurement is available at each time instant $k$.
The system dynamics are subject to individual state chance constraints 
\begin{align}
	\mathbb{P}(h_i^\top x(k) \leq 1) \geq p^i_{x} \quad i \in \{1, \ldots, r\} \label{eq:state_constraints}
\end{align}
and hard input constraints
\begin{align}
	&u(k) \in \mathbb{U}, \label{eq:input_constraints}
\end{align}
where $\mathbb{U} \subseteq \mathbb{R}^m$. Consider a cost function $V(x,u)$, where $V: \mathbb{R}^{n (N_T + 1)} \times \mathbb{R}^{m N_T } \to \mathbb{R}_{\geq 0}$. We aim to solve the following finite horizon stochastic optimal control problem (SOCP) over a task horizon $N_T \in \mathbb{N}$
\begin{subequations}
	\begin{align}
		\!\min_{x,u} & \quad \mathbb{E}_{\mathbb{P}} \bigg \{  V\big(x(0, \ldots, N_T), u(0, \ldots, N_T-1) \big) \bigg \} \label{eq:S_OPC:cost}\\
		\text{s.t.} & \quad x(k+1) = A x(k) + B u(k) + \bar{w}(k) + w(k) \label{eq:S_OPC:dynamic} \\
		& \quad  	\mathbb{P}(h_i^\top x(k) - 1\leq 0) \geq p^i_{x} \quad i \in \{1, \ldots, r\} \label{eq:S_OPC:state_constr}\\
		& \quad u(k) \in \mathbb{U}\\
		& \quad \bar{w}(k) \in \bar{\mathcal{W}}, \quad W = [w(0), \ldots, w(N_T-1)] \sim \mathbb{P} \\
		& \quad x(0) = x_0 \label{eq:S_OPC:init}
	\end{align}
	\label{eq:S_OPC}%
\end{subequations}
for all $k \in \{0, \ldots, N_T-1\}$. 
Unfortunately, problem \eqref{eq:S_OPC} contains several sources of intractability, that is (i) the expectation in \eqref{eq:S_OPC:cost} is taken w.r.t. the true (unknown) probability measure $\mathbb{P}$, (ii) the chance constraints \eqref{eq:S_OPC:state_constr}  must again be evaluated under the true (unknown) measure $\mathbb{P}$ and (iii) optimizing over general control policies $u$ in the presence of possibly unbounded disturbances $w$ yields an infinite dimensional problem.
\subsection{Distributionally Robust Optimization}
In this section we introduce concepts from DRO to reformulate the SOCP, such that the intractability sources (i) and (ii) can be cast into tractable surrogates.

We follow a data-driven approach and assume the existence of a (possibly small) amount of data. Furthermore, we include the case of time correlated disturbances.
\begin{assum}[Distributional assumptions]~ 
	\label{assum:disturbance_set} 
	\begin{enumerate}
		\item The true probability distribution $\mathbb{P}$ is light-tailed.
		\item There exists a dataset $\hat{\mathcal{W}} = \{ \hat{w}_j \}_{j=1}^{\bar{N}}$ that consists of $\bar{N} \in \mathbb{N}$ disturbance trajectories $\hat{w}_j = [\hat{w}_j(0), \ldots, \hat{w}_j(N_T-1)]^\top \sim \mathbb{P}$.
	\end{enumerate}
\end{assum}
A straight forward approach to solve \eqref{eq:S_OPC} is to evaluate \eqref{eq:S_OPC:cost} and \eqref{eq:S_OPC:state_constr} with the empirical measure $\mathbb{\hat{P}} = \delta_{\hat{\mathcal{W}}_{N_s}}$, where $\delta_{\hat{\mathcal{W}}_{N_s}}$ is the Dirac delta measure concentrated on $N_s$ samples of $\hat{\mathcal{W}}$. In other words, we approximate the SOCP with a sample-average approximation, which we denote as the SAA-OCP. The result of the SAA-OCP provides an optimal input sequence $\hat{u}^*(\cdot)$ that minimizes the in-sample performance, i.e. the expected cost in terms of $\hat{\mathbb{P}}$, while the chance constraints are only empirically verified.
If we apply $\hat{u}^*(\cdot)$ to \eqref{eq:dynamic}, then the dynamics are affected by new disturbances $w$ that may not be captured by the dataset $\hat{\mathcal{W}}_{N_s}$. Therefore, the input sequence $\hat{u}^*(\cdot)$ may show a poor out-of-sample performance (expected cost w.r.t. $\mathbb{P}$) and furthermore, the true chance constraints \eqref{eq:S_OPC:state_constr} might be violated.
\begin{remark}
	\label{rem:SSA}
	The optimizer of the SAA-OCP converges almost surely to the optimizer of \eqref{eq:S_OPC} when $N_s$ tends to infinity, whereas for small $N_s$ the SAA control input $\hat{u}^*(\cdot)$ performs poorly when applied to the real system \eqref{eq:dynamic}. Unfortunately, the sample size cannot be chosen arbitrarily large, since the sample complexity of the SAA-OCP increases at least linearly in the sample size $N_s$, which ultimately boils down to a trade-off between accuracy and computational effort \cite{kleywegt2002sample}. This is our main motivation to study distributionally robust SOCPs that allow us to derive meaningful control inputs from a small sample size $N_s$ such that the system states satisfy the chance constraints.
\end{remark}
To robustify the optimization problem against distributional uncertainty we follow \cite{esfahani2018data} and introduce an ambiguity set in terms of the Wasserstein metric defined on the space $\mathcal{M}(\Xi)$, which denotes a set of all probability distributions $\mathbb{Q}$ supported on $\Xi$ with $\mathbb{E}_{\mathbb{Q}} \{ \Vert w \Vert_q \} < \infty$.
\begin{definition}
	\label{def:wasserstein_metric}
	Let $q \in [1, \infty]$. The $q$-Wasserstein metric $d^q_W(\mathbb{Q}_1, \mathbb{Q}_2) : \mathcal{M}(\Xi) \times \mathcal{M}(\Xi) \rightarrow \mathbb{R}_{\geq 0}$ is defined as
	\begin{align*}
		d^q_W(\mathbb{Q}_1, \mathbb{Q}_2) \coloneqq \inf_\Pi \bigg \{  \int_{\Xi^2} \Vert w_1 - w_2 \Vert_q \Pi(d w_1, d w_2) \bigg \},
	\end{align*}
	where $\Pi$ is a joint distribution of $w_1$ and $w_2$ with marginal distributions $\mathbb{Q}_1 \in \mathcal{M}(\Xi)$ and $\mathbb{Q}_2 \in \mathcal{M}(\Xi)$.
\end{definition} 
The Wasserstein metric measures distances between probability distributions by solving an optimal mass transport problem, where the shortest distance is characterized by the optimal transport plan $\Pi$. 
\begin{definition}
	\label{def:wasserstein_ball}
	The Wasserstein ambiguity set centered at the distribution $\mathbb{Q}$ with radius $\epsilon \geq 0$ is given by
	\begin{align*}
		\mathbb{B}_\epsilon(\mathbb{Q}) \coloneqq \{ \mathbb{Q}' \in \mathcal{M}(\Xi) \: | \: d^q_W(\mathbb{Q}, \mathbb{Q}') \leq \epsilon \}.
	\end{align*}
\end{definition}
The stochastic control problem \eqref{eq:S_OPC} can now be robustified against sampling errors for both the cost function \eqref{eq:S_OPC:cost} and the chance constraint \eqref{eq:S_OPC:state_constr} by considering the worst-case distribution over the Wasserstein ball, that is
\begin{align}
	\sup_{\mathbb{Q} \in \mathbb{B}_\epsilon(\mathbb{\hat{P}})} \mathbb{E}_{\mathbb{Q}} \bigg \{  V\big(x(0, \ldots, N_T), u(0, \ldots, N_T-1) \big)  \bigg \} \label{eq:DR_cost_intro}
\end{align}
and
\begin{align}
	\!\inf_{\mathbb{Q} \in \mathbb{B}_\epsilon(\mathbb{\hat{P}})} \mathbb{P}(h_i^\top x(k) \leq 1) \geq p^i_{x}. \label{eq:DR-VaR}
\end{align}
Based on the concentration inequality result \cite[Thm. 3.4]{esfahani2018data}, one can find an optimal Wasserstein radius $\epsilon$ satisfying the following assumption, see also \cite[Thm. 3.5]{esfahani2018data}.
\begin{assum}
	\label{assum:Wasserstein}
	For a given confidence level $\beta \in (0,1)$ and sample size $N_s \leq \bar{N}$ there exists a Wasserstein radius $\epsilon(\beta, N_s)$, such that 
	\begin{align*}
		\mathbb{P}^{N_s}( \mathbb{P} \in \mathbb{B}_\epsilon( \hat{\mathbb{P}} ) ) \geq 1 - \beta. \footnotemark
	\end{align*}
\end{assum}
\footnotetext{$\mathbb{P}^{N_s}$ denotes the $N_s$-fold product distribution.}
The following result provides a tractable reformulation of worst-case expectation problems of type \eqref{eq:DR_cost_intro}.
\begin{lemma}
	\label{lem:convex}
	Assume that $f(\xi)$ is proper, convex and lower semicontinuous, $\xi \in \Xi = \mathbb{R}^n$ and let $q \in [1, \infty]$. Then it holds that
	
	{\small
	\begin{align*}
		\sup_{\mathbb{Q} \in \mathbb{B}_\epsilon(\mathbb{\hat{P}})} \mathbb{E}_{\mathbb{Q}}  \{f(\xi) \} = \inf_{\lambda \geq 0} \lambda \epsilon + \frac{1}{N_s} \sum_{j=1}^{N_s}  \sup_{\xi \in \mathbb{R}^n} \{ f(\xi) - \lambda \Vert \xi - \xi_j \Vert_q \}.
	\end{align*}
	}
\end{lemma}
\begin{proof}
	The proof follows from \cite[Theorem 4.2]{esfahani2018data} and relies on marginalizing and dualizing the Wasserstein constraint $\mathbb{Q} \in \mathbb{B}_\epsilon(\mathbb{\hat{P}})$.
\end{proof}
\section{Indirect feedback DR-MPC}
\label{sec:DR-Smpc}
In this section we address the third source of intractability and approximate the stochastic control problem \eqref{eq:S_OPC} in a receding horizon fashion over a prediction horizon $N$ with $N \ll N_T$. We follow an indirect feedback tube-based approach \cite{hewing2020recursively} and split the dynamics \eqref{eq:dynamic} into a nominal and error part, such that $x(k) = z(k) + e(k)$. Analogously, we separate the input $u(k)$ into a nominal part $v(k)$ and an error part $e_u(k) = \pi(e(k))$, where $\pi(\cdot)$ is the tube controller, so that $u(k) = v(k) + \pi(e(k))$. The resulting decoupled closed-loop dynamics are
\begin{align*}
	z(k+1) &= A z(k) + B v(k) + \bar{w}(k)  \\
	e(k+1) &= A e(k) + B \pi(e(k)) + w(k)
\end{align*}
with initial conditions $z(0) = x(0)$ and $e(0) = 0$. To make predictions at time $k$, we define the $t$-step predictive dynamics
\begin{subequations}
\begin{align}
	x(t+1|k) &= A x(t|k) + B u(t|k) + \bar{w}(t|k) + w(t|k) \nonumber\\
	z(t+1|k) &= A z(t|k) + B v(t|k) + \bar{w}(t|k) \label{eq:pred_state_dyn} \\
	e(t+1|k) &= A e(t|k) + B \pi(e(t|k)) + w(t|k), \label{eq:pred_error_dyn} \\
	u(t|k) &= v(t|k) + e_u(t|k) = v(t|k) + \pi(e(t|k)), \label{eq:pred_tube}
\end{align}
\end{subequations}
which are coupled to the closed-loop dynamics with $x(0|k) = x(k)$, $z(0|k) = z(k)$, $e(0|k) = e(k)$ and the known disturbance part satisfies $\bar{w}(t|k) = \bar{w}(k+t)$. The predictive disturbance sequence $W(k) = [w(0|k), \ldots, w(N-1|k)]^\top$ is obtained by conditioning $W$ on all past disturbances, such that 

{\small
\begin{align*}
	\mathbb{P}_{W(k)} = \mathbb{P}\bigg( [w(k), \ldots, w(k+N-1)]^\top \bigg| [w(0), \ldots, w(k-1)]^\top\bigg).
\end{align*}
}
\subsection{Objective function}
\label{sec:objective}
In the following we approximate the cost function \eqref{eq:DR_cost_intro} over a shortened prediction horizon $N$, i.e.
\begin{align}
	\!\sup_{\mathbb{Q} \in \mathbb{B}_\epsilon(\mathbb{\hat{P}}_{W(k)})} \mathbb{E}_{\mathbb{Q}} \bigg \{ V_f(x(N|k)) + l_1(x(\cdot|k)) + l_2(u(\cdot|k)) \bigg \}, \label{eq:dr_cost}
\end{align}
where $V_f: \mathbb{R}^n\to \mathbb{R}$ is the terminal cost function that approximates the finite horizon tail for $t = \{N, \ldots, N_T\}$, $l_1: \mathbb{R}^{n N} \to \mathbb{R}_{\geq 0}$ denotes the state and $l_2: \mathbb{R}^{m N} \to \mathbb{R}_{\geq 0}$ the input cost function. Moreover, $\mathbb{\hat{P}}_{W(k)}$ is the empirical predictive distribution, which approximates the unknown conditional predictive distribution $\mathbb{P}_{W(k)}$
with disturbance sequences obtained from $\hat{\mathcal{W}}$ (Assumption \ref{assum:disturbance_set}).
To this end, we select $N_s$ trajectories from $\hat{\mathcal{W}}$ starting at time $k$, such that
\begin{align*}
	\hat{\mathcal{W}}_{N_s} = \big \{ [\hat{w}_j(k), \ldots, \hat{w}_j(k + N - 1)]^\top \big \}_{j=1}^{N_s}
\end{align*}
where the empirical distribution is given by $\mathbb{\hat{P}}_{W(k)} = \delta_{\hat{\mathcal{W}}_{N_s}}$.

\subsubsection{Nonlinear tube controllers} 
\label{sec:nonlinear_tube}
A nonlinear tube controller $\pi(\cdot)$ does not allow for an explicit representation of the error \eqref{eq:pred_error_dyn} and input error \eqref{eq:pred_tube} as affine functions of $w$, i.e.
\begin{align*}
	{e}(t|k) &= A^{t} e(0|k) + \sum_{i = 0}^{t-1} A^{t-1-i} [{e}_{u}(i|k) + {w}(k + i)]\\
	{e}_{u}(t|k) &= \pi({e}(t|k)).
\end{align*}
Due to the non-convex mapping between ${e}_{u}$ and $w$, it is therefore not possible to represent the worst-case expectation \eqref{eq:dr_cost} together with Lemma \ref{lem:convex} as a tractable convex optimization problem.
Thus, we approximate the expected value in \eqref{eq:dr_cost} with the empirical predictive distribution, resulting in the sample-average approximation
\begin{align}
	\frac{1}{N_s} \sum_{j=1}^{N_s}  \bigg \{ V_f( \hat{x}_j(N|k)) + l_1(\hat{x}_j(\cdot|k)) + l_2(\hat{u}_j(\cdot|k)) \bigg \}, \label{eq:SAA_cost}
\end{align}
where $\hat{x}_j(t|k) = z(t|k) + \hat{e}_j(t|k)$ and $\hat{u}_j(t|k) = v(t|k) + \hat{e}_{j,u}(t|k)$ are the data dependent states and inputs resulting from $\hat{w}_j(t|k)$.
\subsubsection{Linear tube controllers}
As a special case we investigate linear tube controllers of the form $\pi(e) = K e$, which enable us to write the error and input error explicitly as
\begin{align}
	{e}(t|k) &= A_K^{t} e(0|k) + \sum_{i = 0}^{t-1} A_K^{t-1-i} {w}(k+i) \label{eq:explicit_error} \\
	{e}_{u}(t|k) &= K A_K^{t} e(0|k) + \sum_{i = 0}^{t-1} K A_K^{t-1-i} {w}(k+i), \label{eq:explicit_input_error}
\end{align}
where $A_K = A + B K$. The disturbance sequence $W(k)$ is similar to the nonlinear case obtained from the conditional predictive distribution $\mathbb{P}_{W(k)}$. As it can be seen from \eqref{eq:explicit_error}-\eqref{eq:explicit_input_error}, the error and input error are both affine functions of the disturbance ${w}$, which allows us to state the following result.
\begin{lemma}
	\label{lem:cost}
	Let the tube controller be a linear map $\pi(e) = K e$, let the functions $V_f,l_1$ and $l_2$ be proper, convex and Lipschitz continuous w.r.t. the $q$-norm and let $\Xi = \mathbb{R}^{nN}$.
	Then, for any $\epsilon \geq 0$ the distributionally robust cost \eqref{eq:dr_cost} and the SAA cost \eqref{eq:SAA_cost} share the same minimizer $(z^*, v^*)$.
\end{lemma}
\begin{remark}
	Lemma \ref{lem:cost} implies that if we implement an MPC with the distributionally robust cost \eqref{eq:dr_cost} and an MPC with the SAA cost \eqref{eq:SAA_cost}, the out-of-sample performance will be equivalent, which was similarly mentioned in \cite[Remark 6.7]{esfahani2018data} for the case of static optimization problems.
\end{remark}
\subsection{DR-CVaR constraints}
Chance constraints, also known as Value-at-Risk (VaR) constraints, render the feasible set of the MPC optimization problem in general non-convex. Since this is an undesirable property, we relax the VaR with the CVaR, which is a coherent risk measure and serves as a convex relaxation of the VaR \cite{nemirovski2007convex}.

Let $\gamma_i(x(t|k)) = h_i^\top x(t|k) - 1$ be an affine loss function that describes the $i$-th halfspace constraint \eqref{eq:S_OPC:state_constr} at time step $k$ predicted $t$-steps ahead. We define the CVaR as
\begin{multline*}
	\text{CVaR}_{p_{x}^i}^{\mathbb{{P}}_{w(t|k)}} \big( \gamma_i(x(t|k)) \big) \\
	\coloneqq \!\inf_{\tau_{i,t} \in \mathbb{R}} \bigg( \tau_{i,t} + \frac{1}{1 - p_{x}^i} \mathbb{E}_{\mathbb{{P}}_{w(t|k)}} \big \{ (\gamma_i( x(t|k) ) - \tau_{i,t} )_+ \big \}  \bigg),
\end{multline*}
which requires the unknown $t$-step conditional predictive distribution
\begin{align*}
	\mathbb{P}_{w(t|k)} = \mathbb{P}\bigg( w(t+k)^\top \bigg| [w(0), \ldots, w(k+t-1)]^\top\bigg).
\end{align*}
Hence, we introduce its distributionally robust counterpart by maximizing the expected value over the Wasserstein ball. To this end, we substitute $x = z + e$ and formulate the ambiguity set by means of the empirical $t$-step predictive error distribution $\mathbb{\hat{P}}_{e(t|k)}$, leading to the distributionally robust CVaR constraint 
\begin{align}
	\label{eq:DR-CVaR}
	\!\sup_{\mathbb{Q} \in \mathbb{B}_\epsilon(\mathbb{\hat{P}}_{e(t|k)})} \text{CVaR}_{p_{x}^i}^{\mathbb{Q}} \big (\gamma_i( z(t|k) + e(t|k)) \big) \leq 0.
\end{align}
The distributionally robust constraint set is then given by

{\small
\begin{align*}
	\mathbb{X}_{\text{CVaR}} \coloneqq 
	\left\{  \bar{z} \  \middle\vert \begin{array}{l}
		\displaystyle \! \!\sup_{\mathbb{Q} \in \mathbb{B}_\epsilon(\mathbb{\hat{P}}_{e(t|k)})}  \text{CVaR}_{p_{x}^i}^{\mathbb{Q}} \big(\gamma_i( z(t|k) + e(t|k)) \big) \leq 0\\
		\forall i = \{1, \ldots, r\}, \forall t \in \{0, \ldots, N-1\}
	\end{array}\right\},
\end{align*}
}
where $\bar{z} = [ z(0|k), \ldots, z(N-1|k)]^\top \in \mathbb{R}^{nN}$. In the following we derive a tractable approximation of $\mathbb{X}_{\text{CVaR}}$ as it contains several infinite dimensional optimization problems.
\begin{lemma}
	\label{lem:cvar}
	Let $\alpha_i = 1 - p_x^i \in (0,1)$ and define $p,q \geq 1$, such that the norm equivalence $1/p + 1/q = 1$ holds, then 
	\begin{align}
		\mathbb{Z} \coloneqq &\left\{  \bar{z} \  \middle\vert \begin{array}{l}
			\exists \tau_{i,t} \in \mathbb{R}, \lambda_{i,t} \in \mathbb{R}_{\geq 0}, s_{i,j,t} \in \mathbb{R}_{\geq 0} \:  \text{s.t.} \\
			-\alpha_i \tau_{i,t} + \epsilon \lambda_{i,t} + \frac{1}{N_s} \sum_{j=1}^{N_s} s_{i,j,t} \leq 0\\
			(\gamma_i(z(t|k) + \hat{e}_j(t|k)) + \tau_{i,t})_+ \leq s_{i,j,t}  \\
			\Vert h_i^\top \Vert_p \leq \lambda_{i,t} \\
			\forall j \in \{1, \ldots, N_s\} \: \forall i \in \{1, \ldots, r\} \\
			\forall t \in \{0, \ldots, N-1\}
		\end{array}\right\} \label{eq:cvar_set_wasserstein} \subseteq \mathbb{X}_{\text{CVaR}} 
	\end{align}
\end{lemma}
\begin{remark}
	The DR-CVaR constraint penalizes the worst-case expected constraint violation above the $1 - p_x^i$-th quantile of $\gamma_i(x)$. Hence, \eqref{eq:DR-CVaR} is a sufficient condition for \eqref{eq:DR-VaR} to hold.
\end{remark}
\begin{remark}	
	For linear tube controllers we already know from \eqref{eq:explicit_error}-\eqref{eq:explicit_input_error} that the error and input error can be written as affine functions of the disturbance $w$. 
	Thus, we could substitute \eqref{eq:explicit_error} into \eqref{eq:cvar_set_wasserstein} and center the Wasserstein ball at the $t$-step empirical predictive distribution $\mathbb{\hat{P}}_{w(t|k)}$ instead of $\mathbb{\hat{P}}_{e(t|k)}$. Then we could simultaneously ensure DR-CVaR constraints (Lemma \ref{lem:cvar}) and distributionally robust performance (Lemma \ref{lem:cost}).
\end{remark}
\subsection{Hard input constraints}
To ensure hard input constraints in case of unbounded disturbances, we limit the control authority of the tube controller $\pi$, e.g. via a saturated LQR \cite{hu2002analysis}. To this end, we make the following assumption
\begin{assum}
	\label{assum:input}
	The tube controller $\pi$ satisfies
	\begin{align*}
		\pi(e) \in \mathcal{E}_u \subset \mathbb{U} \: \forall e \in \mathbb{R}^n.
	\end{align*}
\end{assum}
Similar to robust tube-based MPC we tighten the original input constraints, i.e. $\mathbb{V} = \mathbb{U} \ominus \mathcal{E}_u$, where $\mathbb{V}$ denotes the nominal input constraint set.
\begin{remark}
	Note that a saturated LQR belongs to the class of nonlinear tube controllers. Thus, if we want to ensure hard input constraints (Assumption \ref{assum:input}), we cannot achieve distributionally robust performance (Section \ref{sec:nonlinear_tube}), and conversely, if we want to ensure distributionally robust performance, we need a linear tube controller, which, however, does not satisfy Assumption \ref{assum:input} because of the unbounded disturbance $w$.
\end{remark}
\subsection{Recursive feasiblity}
Following the indirect feedback paradigm, we initialize the nominal states and prediction errors with $z(0|k) = z(k)$ and $\hat{e}_j(0|k) = e(k)$ for all $j \in \{1, \ldots, N_s\}$. However, since $e(k)$ is a deterministic variable, the first step constraint at time $t = 0$ may become infeasible. In fact, since we do not assume any bound on the disturbance, it is impossible to guarantee robust recursive feasibility. 

To this end, we introduce a vector $\Theta = [\theta_0, \ldots, \theta_{N-1}]$ consisting of slack variables $\theta_t \geq 0$ for all $t \in \{0, \ldots, N-1\}$ to soften the CVaR constraint set \eqref{eq:cvar_set_wasserstein}, i.e.
\begin{align*}
	\mathbb{Z}_\Theta \coloneqq &\left\{  \bar{z} \  \middle\vert \begin{array}{l}
		\exists \tau_{i,t} \in \mathbb{R}, \lambda_{i,t} \in \mathbb{R}_{\geq 0}, s_{i,j,t} \in \mathbb{R}_{\geq 0} \:  \text{s.t.} \\
		-\alpha_i \tau_{i,t} + \epsilon \lambda_{i,t} + \frac{1}{N_s} \sum_{j=1}^{N_s} s_{i,j,t} \leq \theta_t\\
		(\gamma_i(z(t|k) + \hat{e}_j(t|k)) + \tau_{i,t})_+ \leq s_{i,j,t}  \\
		\Vert h_i^\top \Vert_p \leq \lambda_{i,t} \\
		\forall j \in \{1, \ldots, N_s\} \: \forall i \in \{1, \ldots, r\} \\
		\forall t \in \{0, \ldots, N-1\}
	\end{array}\right\}
\end{align*}
and define an exact penalty function $l_\Theta(\Theta) = c \Vert \Theta \Vert_\infty$ with a sufficiently large penalty weight $c > 0$, see \cite[Thm. 1]{kerrigan2000soft} for details.
In order to ensure stability of the control scheme, we state the following assumption.
\begin{assum}
	\label{assum:terminal}
	There exists a robust positive invariant terminal set $\mathbb{Z}_f$ with respect to $\bar{w} \in \bar{\mathcal{W}}$ under the local controller $\pi_f(z) \in \mathbb{V}$ for all $z \in \mathbb{Z}_f$, i.e.
	\begin{align*}
		\forall z \in \mathbb{Z}_f \Rightarrow Az + B \pi_f(z) + \bar{w} \in \mathbb{Z}_f
	\end{align*}
	and furthermore
	\begin{align*}
		\forall z \in \mathbb{Z}_f : \!\sup_{\mathbb{Q} \in \mathbb{B}_\epsilon(\mathbb{\hat{P}})} \text{CVaR}_{p_{x}^i}^{\mathbb{Q}} \big(\gamma_i(z+ \hat{e}_j(k)) \big) \leq 0
	\end{align*}
	for all $i \in \{1, \ldots, r\}, j \in \{1, \ldots, \bar{N} \}, k \in \{0, \ldots, N_T \}$.
\end{assum}
\begin{remark}
	A set satisfying Assumption \ref{assum:terminal} can be found with methods proposed in \cite{MARK20207136}, i.e. the DR-CVaR can be expressed as a Distributionally Robust PRS, whereas the robust positive invariance w.r.t. $\bar{\mathcal{W}}$ can be established using standard procedures from robust MPC \cite{mayne2005robust}.
\end{remark}
\subsection{Tractable MPC optimization problem}

At each time step $k \geq 0$ we solve the following DR-MPC optimization problem
\begin{align}
	\!\min_{z,v, \tau, \lambda, \Theta} &  \quad l_\Theta(\Theta) + \frac{1}{N_s} \sum_{j=1}^{N_s} V_j(\hat{x}_j, \hat{u}_j) \label{eq:DR_MPC} \\
	\text{s.t.} & \quad \begin{aligned}
		&\hat{x}_j(t+1|k) &&= z(t+1|k) + \hat{e}_j(t+1|k) \\
		&\hat{u}_j(t|k) &&= v(t|k) + \pi( \hat{e}_j(t|k) ) \\
		&\hat{e}_j(t+1|k) &&= A \hat{e}_j(t|k) + B \pi(\hat{e}_j(t|k)) + \hat{w}_j(t|k) \\
		&z(t+1|k) &&= A z(t|k) + B v(t|k) + \bar{w}(t|k) \\
	\end{aligned} \nonumber \\
	& \quad \scalebox{0.978}{$[z(0|k), \ldots z(N|k)] \in \mathbb{Z}_\Theta \times \mathbb{Z}_f, \quad \theta_t \geq 0$} \nonumber \\
	& \quad  v(t|k) \in \mathbb{V} \nonumber\\
	& \quad z(0|k) = z(k), \: \hat{x}_j(0|k) = x(k), \: \hat{e}_j(0|k) = e(k) \nonumber 
\end{align}
for all $t \in \{0, \ldots, N-1\}$ and for all $j \in \{1, \ldots, N_s\}$, where {\small$V_j(\hat{x}_j, \hat{u}_j) = V_f( \hat{x}_j(N|k)) + l_1(\hat{x}_j(\cdot|k)) + l_2(\hat{u}_j(\cdot|k))$}. The control input applied to system \eqref{eq:dynamic} is obtained by getting the first element of the minimizer in \eqref{eq:DR_MPC}, i.e. $v^*(0|k)$ in conjunction with the tube controller
\begin{align}
	u(k) = v^*(0|k) + \pi(e(k)). \label{eq:control_input}
\end{align}
Problem \eqref{eq:DR_MPC} defines a set of feasible control sequences {\small$\mathcal{V}_N(z(k), \Theta) = \{v(\cdot|k), \Theta \: | z(0|k) = z(k), z(\cdot|k) \in \mathbb{Z}_{\Theta} \times \mathbb{Z}_f, v(\cdot|k) \in \mathbb{V} \times \cdots \times \mathbb{V} \}$}, feasible initial states {\small$\mathbb{Z}_N = \{z \: | \: \exists \Theta :\mathcal{V}_N(z, \Theta) \not= \emptyset \}$} and strictly feasible initial states {\small$\mathbb{Z}_N^s = \{z \: | \: \Vert \Theta \Vert = 0, \mathcal{V}_N(z, \Theta) \not= \emptyset \}$}. The strictly feasible initial states correspond to the hard constrained set \eqref{eq:cvar_set_wasserstein}. 
\begin{theorem}
	\label{thm:feasible}
	Let Assumptions \ref{assum:disturbance_set}, \ref{assum:Wasserstein}, \ref{assum:input}, \ref{assum:terminal}, hold and consider system \eqref{eq:dynamic} under control law \eqref{eq:control_input} resulting from \eqref{eq:DR_MPC}. 
	If $x(0) \in \mathbb{Z}_N$, then
	\begin{enumerate}
		\item The MPC optimization problem \eqref{eq:DR_MPC} is recursively feasible for all $0 \leq k \leq N_T - N$
		\item The resulting input $u(k)$ satisfies the hard input constraints \eqref{eq:input_constraints} 
		\item The predicted state sequence $x(\cdot|k)$ satisfies the constraints \eqref{eq:state_constraints} with a probability of at least $1 - \beta$.
	\end{enumerate}
\end{theorem}

\section{Numerical example}
\label{sec:example}
\begin{figure}
	\centering
	\includegraphics[width=1\linewidth]{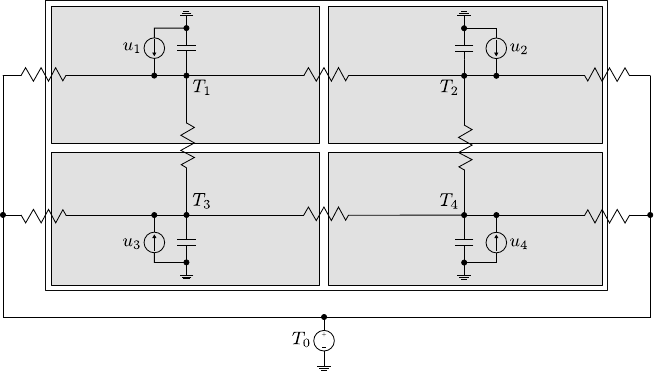}
	\caption{Representation of a four room building model as a resistor-capacitor network with states (temperatures) $T_i$ and inputs (heating/cooling power) $u_i$.}
	\label{fig:schematic}
\end{figure}
As an example we consider a temperature regulation task of a four room building model taken from \cite{hewing2020recursively}. The system is depicted in Figure \ref{fig:schematic} with a state vector $x = [T_1, T_2, T_3, T_4]$, where $T_i$ denotes the room temperature of each room $i = \{1, \ldots, 4\}$. The input vector $u$ consists of the four heat flows, i.e. the heating/cooling power of each HVAC unit, whereas the ambient temperature is given by $T_0$. We consider the following dynamics
\begin{align*}
	x(k+1) = A x(k) + B u(k) + B_w \bar{w}(k) + w(k),
\end{align*}
where the model parameters are taken from \cite{hewing2020recursively}. We model the mean ambient temperature as $\bar{w}(k) = 5 \sin((k+6)/4) + 19$ and the stochastic part as a zero-mean Gaussian process $w \sim \mathcal{N}(0, \Sigma)$, where 
$\Sigma_{ij} = 0.1 + 2 \exp( -(i-j)^2/60 )$ for all $i,j \in \{0,\ldots, N_T \}$. In total we collect $\bar{N} = 1000$ disturbance trajectories (Assumption \ref{assum:disturbance_set}).

The system is subject to hard input constraints on the cooling/heating power $\Vert u \Vert_\infty \leq 4.5$ and individual chance constraints on the room temperature
\begin{align*}
	&\mathbb{P}( x_i(k) \geq 20.4) \geq 0.9 \quad \forall i \in \{1, \ldots, 4\}, \\
	&\mathbb{P}( x_i(k) \leq 21.6) \geq 0.9 \quad \forall i \in \{1, \ldots, 4\}. 
\end{align*}
Starting from the initial condition $x(0) = [20.75 \quad  20.50  \quad 20.65  \quad 20.60]^\top$ we regulate the system to the setpoint $x_s = [21 \quad  21  \quad 21  \quad 21]^\top$ over a run-time of $N_T = 48$ hours.

{\textbf{Simulation setup}:}
We consider a stage cost composed of a weighted 2-norm for the states and a 1-norm for the control input $l_1(x) = \Vert x(\cdot|k) - x_s \Vert_Q$, $l_2(u) = R \Vert u(\cdot|k) \Vert_1$, where $Q = 0.01 I$ and $R = I$. For the CVaR constraints we set $\alpha_i = 1 - p_x^i = 0.3$ for all $i \in \{1, \ldots, 4\}$ and the prediction horizon is $N = 12$. For the Wasserstein constraint we select the $q$-norm as the $1$-norm, which implies that the $p$-norm is the $\infty$-norm.
We design the tube controller $\pi(\cdot)$ as an LQR with weights $Q_\pi = 10^3 I$ and $R_\pi = I$, which we saturated at $\pm 1$. For simplicity we set the terminal set to $\mathbb{Z}_f = \{ x_s \}$ and the terminal cost to $V_f(x) = 0$.

\textbf{Results}:
We carried out $1000$ Monte-Carlo simulations of the system with different noise realizations, $300$ are shown in Figure \ref{fig:temps}. It can be seen that for each disturbance realization, the hard input constraints are satisfied. In Table \ref{table} we compare for different Wasserstein radii $\epsilon$ and sample sizes $N_s$ the resulting worst-case empirical constraint satisfaction (largest in-time constraint violation). 
\begin{center}
	\begin{tabular}{llllll}
		\toprule
		$\epsilon$ & $N_s = 10$ & $N_s = 20 $ & $N_s = 50$ \\ \midrule
		$ 0 $      & $87.8 \% $ & $89.1 \%$ & $91.2 \%$  \\
		$ 10^{-5}$ & $89.5 \% $ & $90.4 \%$ & $92.5 \%$  \\ 
		$ 10^{-4}$ & $91.3 \% $ & $92.2 \%$ & $94.6 \%$   \\
		$10^{-3}$  & $93.1 \% $ & $95.2 \%$ & $97.1 \%$    \\
		\bottomrule
	\end{tabular} 
	\captionof{table}{Impact of Wasserstein radius and sample size on constraint satisfaction of $x_2 \geq 20.4$}
	\label{table}
\end{center}
It can be seen that with $\epsilon = 0$ the chance constraints are empirically violated for $N_s = 10$ and $20$, which underlines the statement that the SAA performs poorly for small sample sizes (Remark \ref{rem:SSA}). The chance constraint satisfaction rate can be increased by either the sample size $N_s$ (higher sample accuracy) or the Wasserstein radius $\epsilon$ (higher robustness against sampling errors). Furthermore, by increasing $N_s$ the Wasserstein radius can be decreased, while maintaining the chance constraint satisfaction level (Assumption \ref{assum:Wasserstein}).
\begin{figure}
	\centering
	\includegraphics[width=0.8\linewidth]{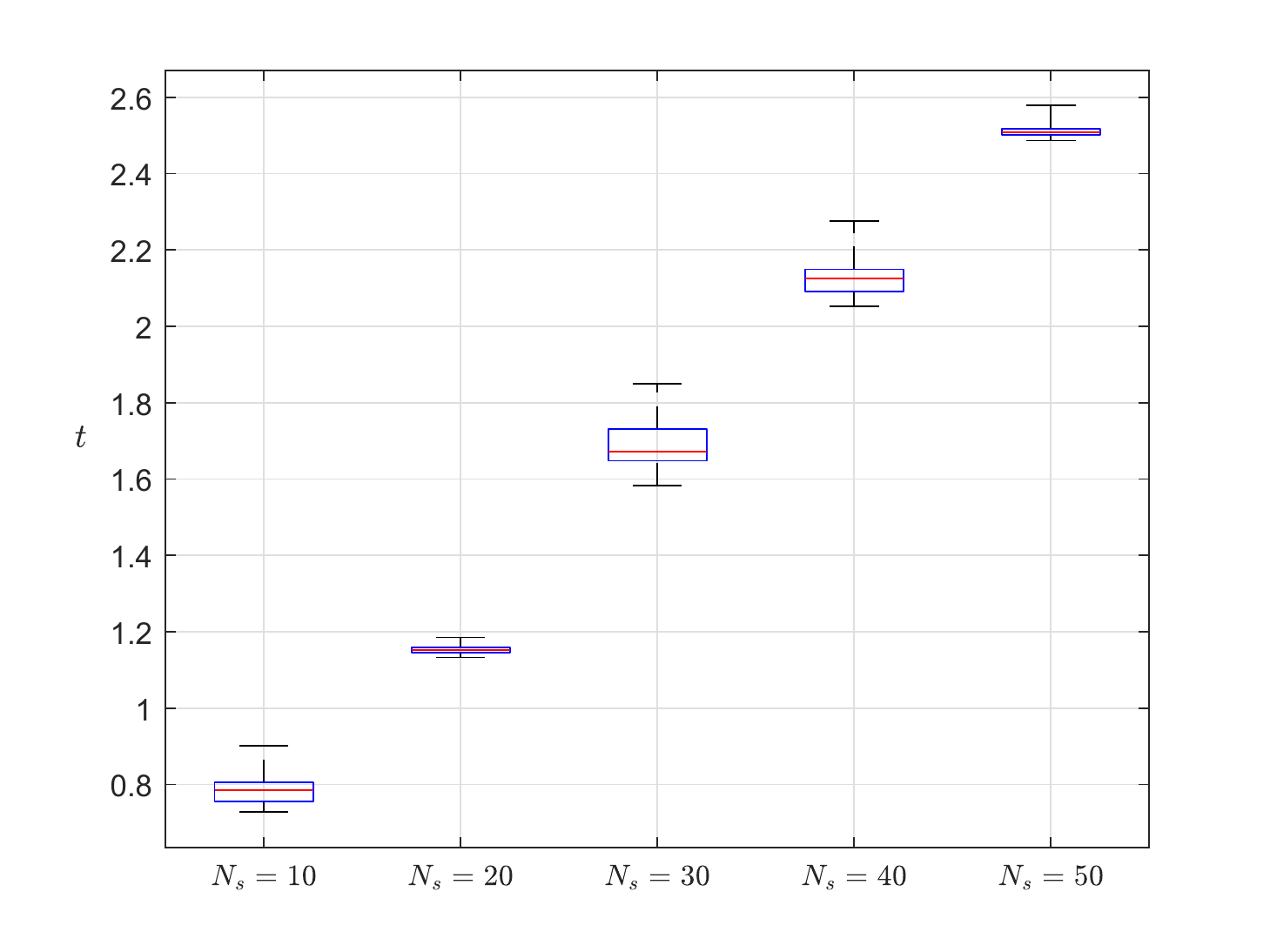}
	\caption{Average runtime $t$ in seconds for different sample sizes $N_s$ for $1000$ Monte-Carlo simulations. For $N_s = 50$ we implemented a SAA-based MPC and for $N_s < 50$ we selected $\epsilon = 10^{-4}$.}
	\label{fig:runtimes}
\end{figure}
In Figure \ref{fig:runtimes} we illustrate the effect of the sample size on the average computation time of the MPC optimization problem. We used CVX \cite{Grant} and ran the simulation on a desktop PC with an Intel i7-9700 CPU and 32gb ram. In order to satisfy the chance constraints empirically via a SAA, we require a sample size of $N_s = 50$, which is on average $3.2$ times slower compared to our distributionally robust approach with $N_s = 10$.
\begin{figure}
	\centering
	\includegraphics[width=0.9\linewidth]{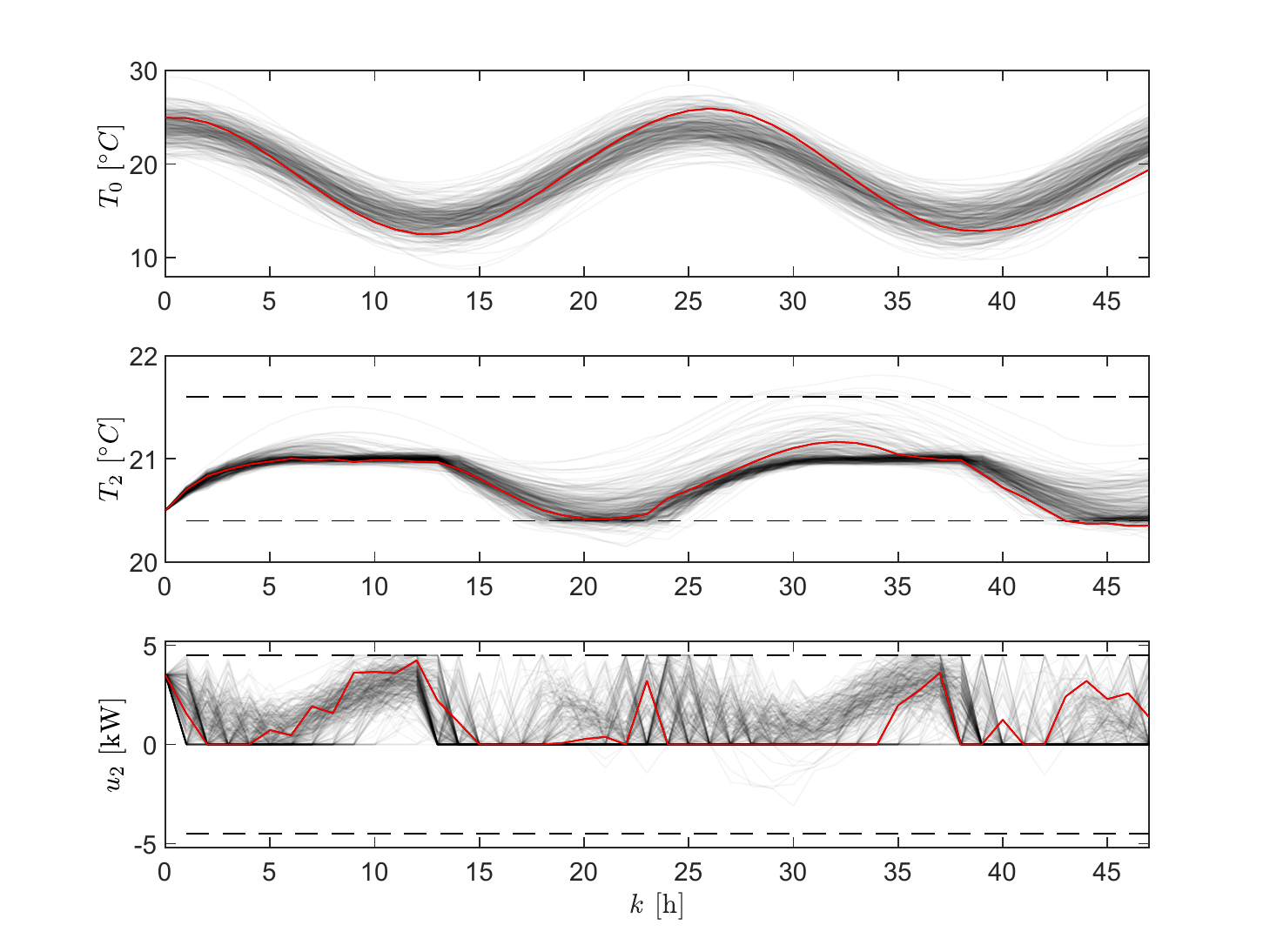}
	\caption{$300$ realizations of the ambient temperature (Top), Room temperature (Middle) and Heating/Cooling power (Bottom). The red lines depict one particular realization.}
	\label{fig:temps}
\end{figure}
\section{Conclusion}
This paper presented a data-driven indirect-feedback DR-SMPC scheme for additive correlated disturbances making use of a scenario-based tube formulation in conjunction with DR-CVaR constraints and Wasserstein ambiguity sets. The soft constrained formulation enabled us to show recursive feasibility, predictive state chance constraint and hard input constraint satisfaction. The effectiveness and computational advantages of our approach were demonstrated on a numerical example of a temperature regulation task.

\section*{Acknowledgement}
The authors would like to thank Peyman Mohajerin Esfahani for his useful comments on the first draft the paper.

\section*{Appendix}
\subsection{Proof of Lemma \ref{lem:cost}}
\label{proof:lemma}
We start by writing the state equation in explicit form under usage of linear superposition
\begin{align}
	x(t|k) &= z(t|k) + e(t|k) \nonumber \\
	&= A^{t} z(0|k) + \sum_{i = 0}^{t-1} A^{t-1-i} \bigg[ B v(i|k) + \bar{w}(i|k) \bigg]  \nonumber \\
	& + A_K^{t} e(0|k) + \sum_{i = 0}^{t-1} A_K^{t-1-i} w(i|k), \label{eq:explicit_states}
\end{align}
where we substituted the explicit form of the nominal state sequence \eqref{eq:pred_state_dyn} and \eqref{eq:explicit_error}. 
Similarly we can express the input equation \eqref{eq:pred_tube} as
\begin{align}
	u(t|k) = v(t|k) + K A_K^{t} e(0|k) + \sum_{i = 0}^{t-1} K A_K^{t-1-i} {w}(i|k). \label{eq:explicit_inputs}
\end{align}
As it can be seen by \eqref{eq:explicit_states}, the state sequence is an affine function of $[v, \bar{w}, w]$ and initial values $[z(0|k), e(0|k)]$, and the input sequence \eqref{eq:explicit_inputs} is an affine function of $[v, w]$ with initial value $e(0|k)$. 

Since the worst-case expectation problem \eqref{eq:dr_cost} is solved w.r.t. the linear disturbance $w$, we define the cost function
\begin{align*}
	\Phi(W(k)) = &V_f(w(0, \ldots, w(N-1|k)) \nonumber \\
	+ &l_1(w(0, \ldots, N-2|k)) \\
	+ &l_2(w(0, \ldots, N-2|k))
\end{align*}
where we neglected the arguments $z, v, e, \bar{w}$ for simplicity. Then we can rewrite \eqref{eq:dr_cost} as
\begin{align}
	\!\sup_{\mathbb{Q} \in \mathbb{B}_\epsilon(\hat{\mathbb{P}}_{W(k)})} \mathbb{E}_{\mathbb{Q}} \bigg \{ &\Phi(W(k)) \bigg \} . \label{eq:proof:dr_cost}
\end{align}
By definition, the functions $V_f$, $l_1$ and $l_2$ are proper, convex and Lipschitz continuous, so is $\Phi(\cdot)$, since it is a sum of nonnegative convex function \cite{boyd2004convex}. This allows us to apply \cite[Theorem 6.3]{esfahani2018data} to \eqref{eq:proof:dr_cost} resulting in
\begin{align}
	\kappa \epsilon + \frac{1}{N_s} \sum_{j=1}^{N_s} \bigg \{ \Phi(\hat{W}_j(k)) \bigg \}, \label{eq:phi_reformulation}
\end{align}
where $\kappa \coloneqq \sup_{\psi \in \mathbb{R}^{nN}} \{ \Vert \psi \Vert_{q,*} \: | \: \Phi^*(\psi) < \infty \}$.
Since $\Phi(\cdot)$ is Lipschitz continuous, there exists a Lipschitz constant $\phi > 0$, which, by \cite[Prop. 6.5]{esfahani2018data} bounds the steepness parameter $\kappa \leq \phi$. This implies that $\kappa$ is independent of any other variable and thus, the term $\kappa \epsilon$ can be neglected when optimizing over $z,v$.

It remains to re-substitute the definition of $\Phi(\hat{W}(k))$ together with the data dependent states and inputs \eqref{eq:explicit_states}-\eqref{eq:explicit_inputs} with $w(i|t) = \hat{w}_j(i|t)$ into \eqref{eq:phi_reformulation}, i.e.
\begin{align*}
	\frac{1}{N_s} \sum_{j=1}^{N_s}  \bigg \{ &V_f( \hat{x}_j(N|k)) + l_1(\hat{x}_j(\cdot|k)) + l_2(\hat{u}_j(\cdot|k))  \bigg \},
\end{align*}
which is equal to \eqref{eq:SAA_cost}. This concludes the proof. $\hfill\blacksquare$

\subsection{Proof of Lemma \ref{lem:cvar}}
Throughout the proof we neglect the time index $t$, half-space constraint index $i$ and abbreviate the $t$-step empirical predictive error distribution as $\hat{\mathbb{P}} = \hat{\mathbb{P}}_{\hat{e}(t|k)}$. By definition of the CVaR we can write \eqref{eq:DR-CVaR} as
\begin{align}
	\!\sup_{\mathbb{Q} \in \mathbb{B}_\epsilon(\mathbb{\hat{P}})} \!\inf_{\tau \in \mathbb{R}} \bigg( -\alpha \tau + \mathbb{E}_{\mathbb{Q}} \big \{ ( \gamma(z + e) + \tau )_+ \big \}  \bigg) \leq 0, \label{eq:proof:dr_cvar}
\end{align} 
which is attained by exchanging $\tau$ with $-\tau$ and multiplying by $\alpha > 0$. By the max-min inequality we can exchange the sup and inf
\begin{align}
	\!\sup_{\mathbb{Q} \in \mathbb{B}_\epsilon(\mathbb{\hat{P}})} \!\inf_{\tau \in \mathbb{R}} \bigg( -\alpha \tau + \mathbb{E}_{\mathbb{Q}} \big \{ ( \gamma(z + e) + \tau )_+ \big \}  \bigg) \nonumber \leq \\
	\!\inf_{\tau \in \mathbb{R}}  -\alpha \tau + \!\sup_{\mathbb{Q} \in \mathbb{B}_\epsilon(\mathbb{\hat{P}})} \mathbb{E}_{\mathbb{Q}} \big \{ (\gamma(z + e) + \tau )_+ \big \}. \label{eq:minmax}
\end{align}
Since $(\gamma(z + e) + \tau)_+ = \text{max}(0, \gamma(z + e) + \tau)$ is the nonnegative pointwise maximum of an affine function, it is proper, convex and lower semicontinuous \cite{boyd2004convex}. Thus, we can apply Lemma \ref{lem:convex} to express the supremum over $\mathbb{Q}$ as
\begin{align*}
	\inf_{\lambda \geq 0} \lambda \epsilon + \frac{1}{N_s} \sum_{j=1}^{N_s}  \sup_{e \in \mathbb{R}^n} \{ ( \gamma(z + e) + \tau )_+ - \lambda \Vert e - \hat{e}_j \Vert_q \}.
\end{align*}
To resolve the max-plus function we replace $\gamma(\cdot)$ with its definition and distinguish between the following two cases. Suppose $ \gamma(z + e) + \tau > 0$, then
\begin{align*}
	&\sup_{e \in \mathbb{R}^n} \big \{ h^\top (z + e) - 1 + \tau - \lambda \Vert e - \hat{e}_j \Vert_q \big \} \\
	=& \sup_{e \in \mathbb{R}^n} \big \{ h^\top ( z + e) - 1 + \tau - \sup_{\Vert \zeta_j \Vert_p \leq \lambda }\zeta_j^\top (e - \hat{e}_j) \big \} \\
	=& \inf_{\Vert \zeta_j \Vert_p \leq \lambda } \big \{h^\top z - 1 + \zeta_j^\top \hat{e}_j + \tau + \sup_{e \in \mathbb{R}^n} \{ (h^\top - \zeta_j^\top) e \}  \big\} \\
	\overset{\zeta_j = h}{=} &h^\top (z + \hat{e}_j) - 1 + \tau  = \gamma(z + \hat{e}_j) + \tau
\end{align*}
where the first equality uses the definition of the dual norm, the second equality the minimax theorem \cite[Prop. 5.5.4]{bertsekas2009convex} and the third equality carries out the supremum over $e$. In the fourth equality, the infimum is dropped because $h$ is not an optimization variable, which additionally requires the constraint $\Vert h \Vert_p \leq \lambda$. On the other hand, if $\gamma(z + e) + \tau \leq 0$, we have $(\gamma(z + e) + \tau)_+ = 0$ and thus
\begin{align*}
	\sup_{e \in \mathbb{R}^n} \{ - \lambda \Vert e - \hat{e}_j \Vert_q \} = \inf_{\Vert \zeta_j \Vert_p \leq \lambda } \sup_{e \in \mathbb{R}^n} \zeta_j^\top (\hat{e}_j -e) \} = 0.
\end{align*} 
In what follows we resort to an epigraph formulation and define for each sample $j$ an auxiliary variable $s_{j}$, such that $s_{j} \geq (  \gamma(z+\hat{e}_j) + \tau )_+$. After combining the above results we arrive at
\[
\eqref{eq:minmax} \leq	\left\{ \begin{array}{l}
	\displaystyle \!\inf_{\tau \in \mathbb{R}, \lambda \geq 0}  -\alpha \tau + \lambda \epsilon + \frac{1}{N_s} \sum_{j=1}^{N_s} s_j \leq 0\\
	\quad \text{s.t.} \quad (\gamma_i(z + \hat{e}_j) + \tau )_+ \leq s_{j}  \\
	\quad \quad \quad \hspace{0.2em} \Vert h \Vert_p \leq \lambda \\
	\quad \quad \quad \hspace{0.2em} \forall j = \{1, \ldots, N_s\}
\end{array}\right.
\]
which, after invoking \eqref{eq:proof:dr_cvar} for all halfspace constraints $i \in \{1, \ldots, r\}$ for each time step $t \in \{0, \ldots, N-1\}$, yields
\begin{align*}
	\left\{  \bar{z} \  \middle\vert \begin{array}{l}
		\displaystyle \inf_{\tau \in \mathbb{R}, \lambda \geq 0} -\alpha_i \tau_{i,t} + \epsilon \lambda_{i,t} + \frac{1}{N_s} \sum_{j=1}^{N_s} s_{i,j,t} \leq 0\\
		(\gamma_i(z(t|k) + \hat{e}_j(t|k)) + \tau_{i,t})_+ \leq s_{i,j,t}  \\
		\Vert h_i^\top \Vert_p \leq \lambda_{i,t} \\
		\forall j \in \{1, \ldots, N_s\} \: \forall i \in \{1, \ldots, r\} \\
		\forall t \in \{0, \ldots, N-1\}
	\end{array}\right\} \subseteq \mathbb{X}_{\text{CVaR}}.
\end{align*}
Similar to \cite[Prop. V.1]{hota2019data} the infimum can be replaced with the existence of variables $\tau, \lambda, s$ satisfying the constraints if and only if the infimum constraint holds true. The "$\Rightarrow$" part can be split into two cases: (i) If the infimum is achieved, then the optimizer satisfies the constraints. (ii) If the infimum is not achieved, then it is $-\infty$ and the first constraint is trivially satisfied. Thus we can find variables $\tau, \lambda, s$ that satisfy the remaining constraints. The "$\Leftarrow$" part is obvious.  \\
\vspace{0.1em} $\hfill \blacksquare$

\subsection{Proof of Theorem \ref{thm:feasible}}
\subsubsection{Recursive feasibility}
Assume that at time $k$ a feasible solution to the MPC problem \eqref{eq:DR_MPC} exists, i.e. $v^*(t|k)$ with slacks $\theta^*(t|k)$ for all $t \in \{0, \ldots, N-1\}$ and states $z^*(t|k)$ for all $t \in \{0, \ldots, N\}$. Applying the control input \eqref{eq:control_input} to system \eqref{eq:dynamic} results in the state $x(k+1)$ and $z(k+1) = z^*(1|k)$, for which we consider the shifted candidate sequence $\tilde{v}(\cdot|k+1) = [v^*(1|k), \ldots, v^*(N-1|k), \pi_f(z^*(N|k))]$ resulting in $\tilde{z}(\cdot|k+1) = [z^*(1|k), \ldots, z^*(N|k), z(N+1|k)]$ with $z(N+1|k) = A z^*(N|k) + B \pi_f(z^*(N|k)) + \bar{w}(k + N)$ and $\tilde{\theta}(\cdot|k+1) = [\theta^*_1, \ldots, \theta^*_{N-1}, 0]$. The zero is appended for the shifted slack variables, since at time $t = N-1$ the state lies in the terminal set.
Since $\tilde{v}(t|k+1) \in \mathbb{V}$ for $t \in \{ 0, \ldots, N-2 \}$ and $\tilde{v}(N-1|k+1) = \pi_f(z^*(N|k)) \in \mathbb{V}$ by Assumption \ref{assum:terminal}, we have that the input constraints are satisfied in prediction. Similarly, we have that $\tilde{z}(0,\ldots, N-1|k+1) \in \mathbb{Z}_{\tilde{\theta}(\cdot|k+1)}$ is satisfied due to the shifted optimal state and slack variables, and $\tilde{z}(N|k+1) = z(N+1|k) \in \mathbb{Z}_f$ by Assumption \ref{assum:terminal}, which verifies the state constraints in prediction.
\subsubsection{Input constraints}
Input constraint satisfaction for $u(k)$ follows immediately from \eqref{eq:control_input}, recursive feasibility and constraint tightening (Assumption \ref{assum:input}), i.e. $u(k) \in \mathbb{U}$, since $\pi(e) \in \mathcal{E}_u \hspace{0.5em} \forall e \in \mathbb{R}^n$.
\subsubsection{Probabilistic state constraint guarantee}
Feasibility at time $k$ implies
\begin{multline*}
	\mathbb{P}_{E(k)}^{N_s} \bigg ( z(0,\ldots, N|k) \in \mathbb{Z}_{\theta(0, \ldots, N-1|k)} \times \mathbb{Z}_f \bigg ) \\
	\geq \mathbb{P}_{E(k)}^{N_s} \bigg ( z(0,\ldots, N|k) \in \mathbb{Z} \times \mathbb{Z}_f \bigg )
	\geq 1 - \beta,
\end{multline*}
where the first inequality holds since $\theta_t \geq 0$ for all $t \in \{0, \ldots, N-1\}$ and the second inequality follows from \cite[Thm. 3.5]{esfahani2018data} (Assumption \ref{assum:Wasserstein}). 
This verifies the DR-CVaR state constraints \eqref{eq:DR-CVaR} for $0 \leq t \leq N$ and for all $i \in \{1, \ldots, r\}$ with a probability of at least $1-\beta$ w.r.t. the $N_s$ fold conditional predictive error distribution $\mathbb{P}^{N_s}_{E(k)}$. By definition the DR-CVaR majorizes the distributionally robust VaR \eqref{eq:DR-VaR}, which implies that the chance constraints of level $p_{x,i}$ are satisfied in prediction with a probability of at least $1-\beta$.
$\hfill\blacksquare$

\end{document}